\documentclass[12pt]{amsart}
\usepackage{amsmath,amssymb,mathrsfs,mathtools,microtype,xspace,bm,bbding,txfonts}
\usepackage[all,cmtip]{xy}
\usepackage[margin=1in]{geometry}
\DeclareTextFontCommand{\emph}{\bfseries\upshape}

\usepackage[colorlinks,allcolors=blue,final,backref=page,linktocpage]{hyperref}

\newtheorem{thm}{Theorem}[section]
\newtheorem{cor}[thm]{Corollary}

\newtheorem{lem}[thm]{Lemma}
\newtheorem{prop}[thm]{Proposition}
\theoremstyle{definition}
\newtheorem{defi}[thm]{Definition}
\newtheorem{example}[thm]{Example}
\newtheorem{rem}[thm]{Remark}
\newtheorem{rmk}[thm]{Remark}

\numberwithin{equation}{section}

\def\Der{\operatorname{Der}}
\newcommand{\dM}{\mathrm{d}}
\newcommand{\Lie}{\mathsf{Lie}}
\newcommand{\Leib}{\mathsf{Leib}}
\newcommand{\Ass}{\mathsf{Ass}}
\newcommand{\preLie}{\mathsf{preLie}}
\newcommand{\gl}{\mathfrak{gl}}

\newcommand{\K}{\mathbf{K}}
\newcommand{\ad}{\mathrm{ad}}

\newcommand{\half}{\frac{1}{2}}

\DeclareMathOperator{\End}{End}
\DeclareMathOperator{\Hom}{Hom}

\def\id{\operatorname{id}}

\newcommand{\nc}{\newcommand}
\newcommand{\delete}[1]{}
\nc{\mlabel}[1]{\label{#1}}  
\nc{\mcite}[1]{\cite{#1}}  
\nc{\mref}[1]{\ref{#1}}  
\nc{\mbibitem}[1]{\bibitem{#1}} 
\delete{
\nc{\mlabel}[1]{\label{#1}{\hfill \hspace{1cm}{\bf{{\ }\hfill(#1)}}}}
\nc{\mcite}[1]{\cite{#1}{{\bf{{\ }(#1)}}}}  
\nc{\mref}[1]{\ref{#1}{{\bf{{\ }(#1)}}}}  
\nc{\mbibitem}[1]{\bibitem[\bf #1]{#1}} 
}
\newcommand{\emptycomment}[1]{}

\delete{
\setlength{\baselineskip}{1.8\baselineskip}

\newcommand{\emptycomment}[1]{}

}

\newcommand{\lon }{\,\rightarrow\,}
\newcommand{\g}{\mathfrak g}
\newcommand{\h}{\mathfrak h}

\newcommand{\B}{\mathsf{B}}
\newcommand{\sfF}{\mathsf{F}}
\newcommand{\G}{\mathsf{G}}
\newcommand{\NR}{\mathsf{NR}}
\newcommand{\MN}{\mathsf{MN}}

\newcommand{\huaL}{\mathcal{L}}
\newcommand{\huaR}{\mathcal{R}}
\newcommand{\frkC}{\mathfrak C}
\newcommand{\frkX}{\mathfrak X}
\newcommand{\frkY}{\mathfrak Y}

\begin{document}

\title[Review of deformation theory I]{Review of deformation theory I: Concrete formulas for deformations of algebraic structures}

\author{Ai Guan}
\address{Department of Mathematics and Statistics, Lancaster University, Lancaster LA1 4YF, UK}
\email{a.guan@lancaster.ac.uk}

\author{Andrey Lazarev}
\address{Department of Mathematics and Statistics, Lancaster University, Lancaster LA1 4YF, UK}
\email{a.lazarev@lancaster.ac.uk}

\author{Yunhe Sheng}
\address{Department of Mathematics, Jilin University, Changchun 130012, Jilin, China}
\email{shengyh@jlu.edu.cn}

\author{Rong Tang}
\address{Department of Mathematics, Jilin University, Changchun 130012, Jilin, China}
\email{tangrong16@mails.jlu.edu.cn}

\date{\today}

\begin{abstract}
In this review article, first we give the concrete formulas of representations and cohomologies of associative algebras, Lie algebras, pre-Lie algebras, Leibniz algebras and 3-Lie algebras and some of their strong homotopy analogues. Then we recall the graded Lie algebras and graded associative algebras that  characterize these algebraic structures as  Maurer-Cartan elements. The corresponding Maurer-Cartan element
equips the graded Lie or associative algebra  with a differential. Then the
deformations of the given algebraic structures are characterized as the
Maurer-Cartan elements of the resulting differential graded Lie or associative
algebras. We also recall the relation between the cohomologies and the differential graded Lie and associative algebras that control the deformations.
\end{abstract}

\subjclass[2010]{13D03,13D10,17B56,17A30}

\keywords{cohomology, deformation, Maurer-Cartan element, associative algebra, Lie algebra, pre-Lie algebra, Leibniz algebra, 3-Lie algebra}

\maketitle
\tableofcontents
\allowdisplaybreaks

\section{Introduction}\mlabel{sec:intr}

This article is the first part of a review of deformation theory of algebraic structures, which concentrates on the concrete formulas of the cohomologies and the differential graded Lie algebras that control deformations of associative algebras, Lie algebras, pre-Lie algebras, Leibniz algebras and 3-Lie algebras. In the second part, we will use a homotopical approach to the abstract deformation theory and show that any reasonable deformation theory is controlled by a differential graded Lie algebra in a precise sense.

Deformation theory started with the seminal  work of Kodaira and Spencer~\cite{KS} on deformation of holomorphic manifolds and then expanded to the realm of algebraic geometry \cite{Ha}; more recently it found striking applications in number theory ~\cite{Maz}. At the same time, the idea of a deformation became prominent in physics; e.g. quantum mechanics is now habitually interpreted as a deformation of classical mechanics and this idea of `deformation quantization' has now fed back to pure mathematics inspiring a host of exciting new developments ~\cite{K1,K2,Ri}.

The subject of \emph{algebraic} deformation theory started with foundational papers
 of Gerstenhaber~\cite{Ge0,Ge,Ge2,Ge3,Ge4} which considered associative algebras as an object to be deformed. Deformations of Lie algebras were studied by
Nijenhuis and Richardson~\cite{NR,NR2}. More recently there has been a lot of activity involving deformations of other algebraic structures of operadic origin (i.e. involving multi-linear operations such as pre-Lie algebras, Leibniz algebras, $n$-Lie algebras and others)
 \cite{Bal,Dz,Figueroa,Loday-P,NR bracket of n-Lie,Tcohomology,WBLS}. Simultaneous deformations of algebras and morphisms are studied in \cite{F-Marco-1,F-Marco-2}. Deformations of
general operadic algebras are considered in
Kontsevich-Soibelman draft book ~\cite{KSo} and in the monograph by Loday-Vallette~\cite{LV}. A categorical approach using triples and cotriples is developed in ~\cite{Fo}.

The modern approach to algebraic deformation theory (at least in characteristic zero)
dictates that a deformation problem should have a differential graded Lie algebra $\g$ `controlling it' in the sense
that  it seeks to deform a Maurer-Cartan element (typically zero) in $\g$.
At the same time, there
should be a cohomology theory so that the infinitesimal of a
formal deformation would correspond, up to a suitable equivalence relation, to a cohomology class, as well as a concomitant obstruction theory.

 Such a cohomology theory for operadic algebras is (generalized) Andr\'e-Quillen cohomology ~\cite{LV}. In the case of associative and Lie algebras it specializes to Hochschild and Chevalley-Eilenberg theory respectively.


\smallskip
Throughout this paper, we work with a coefficient field $\K$ which is  of characteristic 0.

\section{The cohomologies of algebraic structures}

In this section, we review representations and cohomologies of associative algebras, Lie algebras, pre-Lie algebras, Leibniz algebras and 3-Lie algebras respectively.

\subsection{Associative algebras}

Recall that an associative algebra  is a vector space $\g$  equipped with a multiplication $\cdot:\g\otimes \g\lon \g$ satisfying the associative law:
$$
x\cdot(y\cdot z)=(x\cdot y)\cdot z,\quad \forall x,y,z\in\g.
$$

\begin{defi}
Let $(\g,\cdot)$ be an associative algebra and $V$   a vector space. Let $\huaL,\huaR:\g\longrightarrow\gl(V)$ be two linear maps with $x\rightarrow \huaL_x$ and $x\rightarrow \huaR_x$ respectively. The triple $(V;\huaL,\huaR)$ is called a \emph{representation} of $\g$ if for all  $x,y\in \g,$ we have
$$\huaL_{x\cdot y}=\huaL_x\circ \huaL_y,\quad \huaR_{x\cdot y}=\huaR_y\circ \huaR_x,\quad \huaL_x\circ \huaR_y=\huaR_y\circ\huaL_x.$$
\end{defi}

\begin{lem}
Let $(V;\huaL,\huaR)$ be a representation of an associative algebra $(\g,\cdot)$.
Define $\huaL^*:\g\longrightarrow \gl(V^*)$ and $\huaR^*:\g\longrightarrow \gl(V^*)$   respectively  by
$$
\langle \huaL^*_x\alpha,v\rangle=\langle \alpha,\huaL_xv\rangle,\quad\langle \huaR^*_x\alpha,v\rangle=\langle \alpha,\huaR_xv\rangle,\quad \forall ~ x\in \g,\alpha\in V^*,v\in V.
$$
Then $(V^*;\huaR^*,\huaL^*)$ is a representation of $(\g,\cdot)$, which is called the \emph{dual representation} of  $(V;\huaL,\huaR)$.
\end{lem}

\begin{example}
Let $(\g,\cdot)$ be an associative algebra. Let $L_x$ and $R_x$ denote the left and right multiplication operator, respectively, that is, $L_xy=x\cdot y,R_yx=x\cdot y$ for any $x,y\in \g$. Then   $(\g;L,R)$ is a representation of $(\g,\cdot)$, which is called the \emph{regular representation}. Furthermore,  $(\g^*;R^*,L^*)$ is also a representation of $(\g,\cdot)$.
\end{example}

Let $(V;\huaL,\huaR)$ be a representation of an associative algebra $(\g,\cdot)$. A linear map $f\in\Hom(\otimes^n\g,V)$ is called an $n$-cochain $(n\ge 0)$. Denote the set of $n$-cochains by $C^n_{\Ass}(\g;V)$, i.e.~$C^n_{\Ass}(\g;V)=\Hom(\otimes^n\g,V)$. The \emph{Hochschild coboundary operator}
$\dM:C^n_{\Ass}(\g;V)\longrightarrow C^{n+1}_{\Ass}(\g;V)$ is defined by
\begin{equation*}
\begin{split}
(\dM f)(x_1,\cdots,x_{n+1}) ={}& \huaL_{x_1} f(x_2,\cdots,x_{n+1})+\sum_{i=1}^n(-1)^if(x_1,\cdots,x_{i-1},x_i\cdot x_{i+1},\cdots,x_{n+1})\\
&+ (-1)^{n+1}\huaR_{ x_{n+1}}f(x_1,\cdots,x_n)
\end{split}
\end{equation*}
for all $f\in C^n_{\Ass}(\g;V)$ and $x_1,x_2,\cdots,x_{n+1}\in \g$. It is proved in~\cite{Hochschild-1,Hochschild-2} that $\dM\circ\dM=0$.

An $n$-cochain $f$ is called \emph{closed}, if $\dM f=0$. An $n$-cochain $f$ is called \emph{exact}, if $f=\dM g$ for some $g\in C^{n-1}_{\Ass}(\g;V)$. We denote the set of closed $n$-cochains by $Z^n_{\Ass}(\g;V)$ and the set of exact $n$-cochains by $B^n_{\Ass}(\g;V)$. Since $\dM\circ\dM=0$, we have $B^n_{\Ass}(\g;V)\subset Z^n_{\Ass}(\g;V)$. We denote by $H^n_{\Ass}(\g;V)=Z^n_{\Ass}(\g;V)/B^n_{\Ass}(\g;V)$, which is called the $n$-th \emph{Hochschild cohomology group} of the associative algebra $(\g,\cdot)$ with the coefficient in the representation $(V;\huaL,\huaR)$.

\subsection{Lie algebras}

A \emph{Lie algebra} is a vector space $\g$ together with a skew-symmetric bilinear map $[\cdot,\cdot]_\g:\wedge^2\g\lon\g$ such that
\begin{equation*}
[x,[y,z]_\g]_\g+[y,[z,x]_\g]_\g+[z,[x,y]_\g]_\g=0,\quad\forall x,y,z\in\g.
\end{equation*}

\begin{defi}
A \emph{representation} of a Lie algebra $(\g,[\cdot,\cdot]_{\g})$ is a pair $(V;\rho)$, where $V$ is a vector space and $\rho:\g\lon\gl(V)$ is a Lie algebra morphism, i.e.~$\rho([x,y]_\g)=[\rho(x),\rho(y)]$ for all $x,y\in\g.$ Here $[\cdot,\cdot]:\wedge^2\gl(V)\lon\gl(V)$ is the commutator Lie bracket on $\gl(V)$, the vector space of linear transformations on $V$.
\end{defi}

\begin{lem}
Let $(V;\rho)$ be a representation of a Lie algebra $(\g,[\cdot,\cdot]_{\g})$. Define $\rho^*:\g\longrightarrow \gl(V^*)$ by
$$
\langle \rho^*(x)\alpha,v\rangle=-\langle \alpha,\rho(x)v\rangle, \quad\forall ~ x\in \g,\alpha\in V^*,v\in V.
$$
Then $(V^*;\rho^*)$ is a representation of $(\g,[\cdot,\cdot]_{\g})$, which is called the \emph{dual representation} of  $(V;\rho)$.
\end{lem}

\begin{example}
Let $(\g,[\cdot,\cdot]_{\g})$ be a Lie algebra. We define $\ad: \g \longrightarrow \gl(\g)$ by
$$\ad_xy:=[x,y]_\g, \quad \forall x,y\in \g.$$
Then $(\g;\ad)$ is a representation of $(\g,[\cdot,\cdot]_{\g})$, which is called the \emph{adjoint representation}. Furthermore, $(\g^*;\ad^*)$ is also a representation of $(\g,[\cdot,\cdot]_{\g})$.
\end{example}

Let $(V;\rho)$ be a representation of a Lie algebra $(\g,[\cdot,\cdot]_{\g})$. A linear map $f\in\Hom(\wedge^n\g,V)$ is called an $n$-cochain. Denote the set of $n$-cochains by $C^n_\Lie(\g;V)$, i.e.~$C^n_\Lie(\g;V)=\Hom(\wedge^n\g,V)$, $(n\ge 0)$.
The \emph{Chevalley-Eilenberg coboundary operator}
$\dM:C^n_\Lie(\g;V)\longrightarrow C^{n+1}_\Lie(\g;V)$ is defined by
\begin{equation*}
\begin{split}
(\dM f)(x_1,\cdots,x_{n+1}) ={}& \sum_{i=1}^{n+1}(-1)^{i+1}\rho(x_{i})(f(x_1,\cdots,\hat{x}_{i},\cdots,x_{n+1}))\\
&+\sum_{1\le i<j\le n+1}(-1)^{i+j}f([x_i,x_j]_\g,x_1,\cdots,\hat{x}_i,\cdots,\hat{x}_j,\cdots,x_{n+1}),
\end{split}
\end{equation*}
for all $f\in C^n_\Lie(\g;V)$ and $x_1,\cdots,x_{n+1}\in\g.$ It is proved in \cite{Hochschild-3} that $\dM\circ\dM=0$.

An $n$-cochain $f$ is called \emph{closed}, if $\dM f=0$. An $n$-cochain $f$ is called \emph{exact}, if $f=\dM g$ for some $g\in C^{n-1}_\Lie(\g;V)$. We denote the set of closed $n$-cochains by $Z^n_\Lie(\g;V)$ and the set of exact $n$-cochains by $B^n_\Lie(\g;V)$. We denote by $H^n_{\Lie}(\g;V)=Z^n_\Lie(\g;V)/B^n_\Lie(\g;V)$, which is called the $n$-th \emph{Chevalley-Eilenberg cohomology group} of the Lie algebra $(\g,[\cdot,\cdot]_{\g})$ with the coefficient in the representation $(V;\rho)$.

\subsection{pre-Lie algebras}

A \emph{pre-Lie algebra} is a pair $(\g,\cdot_\g)$, where $\g$ is a vector space and  $\cdot_\g:\g\otimes \g\longrightarrow \g$ is a bilinear multiplication
satisfying that for all $x,y,z\in \g$, the associator
$$(x,y,z):=(x\cdot_\g y)\cdot_\g z-x\cdot_\g(y\cdot_\g z)
$$
is symmetric in $x,y$, that is,
$$(x,y,z)=(y,x,z)\;\;{\rm or}\;\;{\rm
equivalently,}\;\;(x\cdot_\g y)\cdot_\g z-x\cdot_\g(y\cdot_\g z)=(y\cdot_\g x)\cdot_\g
z-y\cdot_\g(x\cdot_\g z).$$

\begin{prop}{\rm (\cite{Bai})}\label{3-pre-Lie}
Let $(\g,\cdot_\g)$ be a pre-Lie algebra. Then the commutator
\begin{equation}
[x,y]_C:=x\cdot_\g y-y\cdot_\g x,\quad\forall x,y\in\g,
\end{equation}
defines a Lie algebra $(\g,[\cdot,\cdot]_C)$, which is called the \emph{sub-adjacent Lie algebra} of $(\g,\cdot_\g)$, and denoted by $\g^c$.  Moreover, $(\g,\cdot_\g)$ is called the \emph{compatible pre-Lie algebra} structure on the Lie algebra $\g^c$.
\end{prop}

\begin{defi}
Let $(\g,\cdot_\g)$ be a pre-Lie algebra and $V$ a vector space. A
\emph{representation} of $(\g,\cdot_\g)$ is a triple $(V;\rho,\mu)$,
where $\rho:\g\longrightarrow \gl(V)$ is a representation of the
sub-adjacent Lie algebra $\g^c$ on $V$ and
$\mu:\g\longrightarrow \gl(V)$ is a linear map satisfying
\begin{equation}
\rho(x)\mu(y)v-\mu(y)\rho(x)v=\mu(x\cdot_\g y)v-\mu(y)\mu(x)v, \quad \forall~x,y\in \g,~ v\in V.
\mlabel{eq:repcond2}
\end{equation}
\end{defi}

Let $(V;\rho,\mu)$ be a representation of a pre-Lie algebra $(\g,\cdot_\g)$. For all $x\in \g,~u\in V,~\xi\in V^*$, define $\rho^*:\g\longrightarrow\gl(V^*)$ and $\mu^*:\g\longrightarrow\gl(V^*)$ by
$$\langle \rho^*(x)(\xi),u\rangle=-\langle\xi,\rho(x)(u)\rangle,\quad \langle \mu^*(x)(\xi),u\rangle=-\langle\xi,\mu(x)(u)\rangle.$$

\begin{lem}\label{dual-rep}{\rm(\cite{Bai0})}
Let $(V;\rho,\mu)$ be a representation of a pre-Lie algebra $(\g,\cdot_\g)$. Then $(V^*;\rho^*-\mu^*,-\mu^*)$ is a representation of $(\g,\cdot_\g)$, which is called the \emph{dual representation} of $(V;\rho,\mu)$.
\end{lem}

\begin{example}
Define the left multiplication $L:\g\longrightarrow\gl(\g)$ and the right multiplication $R:\g\longrightarrow\gl(\g)$ by $L_xy=x\cdot_\g y$ and $R_xy=y\cdot_\g x$ respectively for all $x,y\in \g$.  Then $(\g;L,R)$ is a representation of $(\g,\cdot_\g)$, which is called the \emph{regular representation}. Furthermore, $(\g^*;L^*-R^*,-R^*)$ is also a representation of $(\g,\cdot_\g)$.
\end{example}

The cohomology complex for a pre-Lie algebra $(\g,\cdot_\g)$ with the coefficient in a representation $(V;\rho,\mu)$ is given as follows \cite{Dz}.
The set of $n$-cochains  $C^n_\preLie(\g;V)$ is given by
$\Hom(\wedge^{n-1}\g\otimes \g,V),\
n\geq 1.$  The coboundary operator $\dM:C^n_\preLie(\g;V)\longrightarrow C^{n+1}_\preLie(\g;V)$ is given by
\begin{equation}
\begin{split}
(\dM f)(x_1, \cdots,x_{n+1})
={}& \sum_{i=1}^{n}(-1)^{i+1}\rho(x_i)f(x_1, \cdots,\hat{x}_i,\cdots,x_{n+1})\\
&+ \sum_{i=1}^{n}(-1)^{i+1}\mu(x_{n+1})f(x_1, \cdots,\hat{x}_i,\cdots,x_n,x_i)\\
&- \sum_{i=1}^{n}(-1)^{i+1}f(x_1, \cdots,\hat{x}_i,\cdots,x_n,x_i\cdot_\g x_{n+1})\\
&+ \sum_{1\leq i<j\leq n}(-1)^{i+j}f([x_i,x_j]_C,x_1,\cdots,\hat{x}_i,\cdots,\hat{x}_j,\cdots,x_{n+1}),
\mlabel{eq:pLcoh}
\end{split}
\end{equation}
for all $x_1,x_2,\cdots,x_{n+1}\in \g$. It is proved in \cite{Dz} that $\dM\circ\dM=0$. The resulting cohomology is denoted by $H^*_{\preLie}(\g;V)$.

\begin{rmk}
Let $(V;\rho,\mu)$ be a representation of
a pre-Lie algebra $(\g,\cdot_\g)$. Then there is a representation of the sub-adjacent Lie algebra $(\g,[\cdot,\cdot]_{C})$ on the vector space $\Hom(\g,V)$ given by
\begin{equation}\label{eq:2.2}
\hat{\rho}(x)(f)(y)=\rho(x)f(y)+\mu(y)f(x)-f(x\cdot_\g y),\quad \forall f\in \Hom(\g,V),~ x,y\in \g.
\end{equation}
Moreover, we define $\Phi:\Hom(\wedge^{n-1}\g,\Hom(\g,V))\lon\Hom(\wedge^{n-1}\g\otimes\g,V)$ by
$$
\Phi(f)(x_1,\cdots,x_{n-1},y)=f(x_1,\cdots,x_{n-1})(y),\quad \forall x_1,\cdots,x_{n-1},y\in\g.
$$
Then $\Phi$ is a cochain map from the cochain complex $(\oplus_nC^n_\Lie(\g^c;\Hom(\g,V)),\dM)$ to the cochain complex $(\oplus_nC^n_\preLie(\g;V),\dM)$, i.e.~we have the following commutative diagram:
$$\xymatrix{
\Hom(\wedge^{n-1} \g^c, \Hom(\g,V))\ar[d]_{\dM}  \ar[r]^{\quad
\Phi} & \Hom(\wedge^{n-1} \g\otimes \g,V) \ar[d]^{\dM}  \\
\Hom(\wedge^n \g^c, \Hom(\g,V))\ar[r]^{\quad
\Phi}   &\Hom(\wedge^n \g\otimes \g,V).              }$$
Thus $\Phi$ induces an isomorphism between the corresponding cohomologies:
\begin{equation}
H^n_{\preLie}(\g;V)\cong H^{n-1}_{\Lie}\big(\g^c;\Hom(\g,V)\big),~\forall n=1,2,\cdots.
\end{equation}
See \cite{Bu,Dz} for more details.
\end{rmk}

\subsection{Leibniz algebras}

A \emph{Leibniz algebra} is a vector space $\g$ together with a bilinear operation $[\cdot,\cdot]_\g:\g\otimes\g\lon\g$ such that
\begin{equation*}
\label{Leibniz}
[x,[y,z]_\g]_\g=[[x,y]_\g,z]_\g+[y,[x,z]_\g]_\g,\quad\forall x,y,z\in\g.
\end{equation*}

A \emph{representation} of a Leibniz algebra $(\g,[\cdot,\cdot]_{\g})$ is a triple $(V;\rho^L,\rho^R)$, where $V$ is a vector space, $\rho^L,\rho^R:\g\lon\gl(V)$ are linear maps such that the following equalities hold for all $x,y\in\g$,
\begin{align}
\label{rep-1}\rho^L([x,y]_{\g}) &= [\rho^L(x),\rho^L(y)],\\
\label{rep-2}\rho^R([x,y]_{\g}) &= [\rho^L(x),\rho^R(y)],\\
\label{rep-3}\rho^R(y)\circ \rho^L(x) &= -\rho^R(y)\circ \rho^R(x).
\end{align}

Let $(V;\rho^L,\rho^R)$ be a representation of a Leibniz algebra $(\g,[\cdot,\cdot]_{\g})$. Define $(\rho^L)^*:\g\longrightarrow \gl(V^*)$ and $(\rho^R)^*:\g\longrightarrow \gl(V^*)$ respectively by
$$
\langle (\rho^L)^*(x)\alpha,v\rangle=-\langle \alpha,\rho^L(x)v\rangle,\quad\langle (\rho^R)^*(x)\alpha,v\rangle=-\langle \alpha,\rho^R(x)v\rangle,\quad \forall ~ x\in \g,\alpha\in V^*,v\in V.
$$

\begin{lem}\label{lem:dualrep}{\rm(\cite{Sheng-Tang})}
Let\/ $(V;\rho^L,\rho^R)$ be a representation of a Leibniz algebra\/ $(\g,[\cdot,\cdot]_{\g})$. Then
$$\big(V^*;(\rho^L)^*,-(\rho^L)^*-(\rho^R)^*\big)$$
is a representation of\/ $(\g,[\cdot,\cdot]_{\g})$, which is called the \emph{dual representation} of\/ $(V;\rho^L,\rho^R)$.
\end{lem}

\begin{example}
Define the left multiplication $L:\g\longrightarrow\gl(\g)$ and the right multiplication $R:\g\longrightarrow\gl(\g)$ by $L_xy=[x,y]_\g$ and $R_xy=[y,x]_\g$ respectively for all $x,y\in \g$.  Then $(\g;L,R)$ is a representation of $(\g,[\cdot,\cdot]_{\g})$, which is called the \emph{regular representation}. Furthermore, $(\g^*;L^*,-L^*-R^*)$ is also a representation of $(\g,[\cdot,\cdot]_{\g})$. Please see \cite{Sheng-Tang} for more details about dual representations of Leibniz algebras and Leibniz bialgebras.
\end{example}

Let $(V;\rho^L,\rho^R)$ be a representation of a Leibniz algebra $(\g,[\cdot,\cdot]_{\g})$.
The set of $n$-cochains is given by  $C^n_\Leib(\g;V)=
\Hom(\otimes^n\g,V),~(n\ge 0)$.  The coboundary operator
$\dM:C^n_\Leib(\g;V)\longrightarrow C^
{n+1}_\Leib(\g;V)$
is defined by
\begin{equation*}
\begin{split}
(\dM f)(x_1,\cdots,x_{n+1}) ={}& \sum_{i=1}^{n}(-1)^{i+1}\rho^L(x_i)f(x_1,\cdots,\hat{x}_i,\cdots,x_{n+1})+(-1)^{n+1}\rho^R(x_{n+1})f(x_1,\cdots,x_{n})\\
&+ \sum_{1\le i<j\le n+1}(-1)^if(x_1,\cdots,\hat{x}_i,\cdots,x_{j-1},[x_i,x_j]_\g,x_{j+1},\cdots,x_{n+1}),
\end{split}
\end{equation*}
for all $x_1,\cdots, x_{n+1}\in\g$. It was proved in \cite{Loday-P} that $\dM\circ\dM=0$. The resulting cohomology is denoted by $H^*_{\Leib}(\g;V)$.


\subsection{$3$-Lie algebras}

\begin{defi}{\rm(\cite{Filippov})}
A \emph{3-Lie algebra}  is a vector space $\g$ together with a skew-symmetric linear map  $[\cdot,\cdot,\cdot]_\g:
\otimes^3 \g\rightarrow \g$ such that the following \emph{Fundamental Identity} holds:
\begin{equation}\label{eq:de1}
\begin{split}
\sfF_{x_1,x_2,x_3,x_4,x_5}
\triangleq{}& [x_1,x_2,[x_3,x_4,x_5]_\g]_\g-[[x_1,x_2,x_3]_\g,x_4,x_5]_\g-[x_3,[x_1,x_2,x_4]_\g,x_5]_\g\\
&- [x_3,x_4,[x_1,x_2,x_5]_\g]_\g\\
={}& 0.
\end{split}
\end{equation}
\end{defi}

Elements in $\wedge^2\g$ are called \emph{fundamental objects} of the $3$-Lie algebra $(\g,[\cdot,\cdot,\cdot]_\g)$. There is a bilinear operation $[\cdot,\cdot]_{\sfF}$ on $  \wedge^{2}\g$, which is given by
\begin{equation}\label{eq:bracketfunda}
[\frkX,\frkY]_{\sfF}=[x_1,x_2,y_1]_\g\wedge y_2+y_1\wedge[x_1,x_2,y_2]_\g,\quad
\forall \frkX=x_1\wedge x_2,~\frkY=y_1\wedge y_2.
\end{equation}
It is well-known that $(\wedge^2\g,[\cdot,\cdot]_{\sfF})$ is a Leibniz algebra \cite{DT}, which plays  an important role in the theory of 3-Lie algebras.

\begin{defi}{\rm (\cite{Kasymov})}\label{defi:usualrep}
A \emph{representation} $\rho$ of a $3$-Lie algebra $(\g,[\cdot,\cdot,\cdot]_\g)$ on a vector space  $V$ is a linear map $\rho:\wedge^2\g\longrightarrow \End(V)$, such that
\begin{align*}
\rho(x_1,x_2)\rho(x_3,x_4) &= \rho([x_1,x_2,x_3]_\g,x_4)+\rho(x_3,[x_1,x_2,x_4]_\g)+\rho(x_3,x_4)\rho(x_1,x_2),\\
\rho(x_1,[x_2,x_3,x_4]_\g) &= \rho(x_3,x_4)\rho(x_1,x_2)-\rho(x_2,x_4)\rho(x_1,x_3)+\rho(x_2,x_3)\rho(x_1,x_4).
\end{align*}
\end{defi}

\emptycomment{
It is straightforward to obtain
\begin{lem}\label{lem:semidirectp}
Let $(\g,[\cdot,\cdot,\cdot]_\g)$ be a $3$-Lie algebra, $V$  a vector space and $\rho:
\wedge^2\g\rightarrow \gl(V)$  a skew-symmetric linear
map. Then $(V;\rho)$ is a representation of $\g$ if and only if there
is a $3$-Lie algebra structure $($called the semidirect product$)$
on the direct sum of vector spaces  $\g\oplus V$, defined by
\begin{equation}\label{eq:sum}
[x_1+v_1,x_2+v_2,x_3+v_3]_{\rho}=[x_1,x_2,x_3]_\g+\rho(x_1,x_2)v_3+\rho(x_2,x_3)v_1+\rho(x_3,x_1)v_2,
\end{equation}
for $x_i\in \g, v_i\in V, 1\leq i\leq 3$. We denote this semidirect product $3$-Lie algebra by $\g\ltimes_\rho V.$
\end{lem}
}

Let $A$ be a vector space. For a linear map $\phi:A\otimes A\lon\gl(V)$, we define a linear map $\phi^*: A\otimes A\lon\gl(V^*)$ by
\begin{equation*}
\langle \phi^*(x,y)\alpha,v\rangle=-\langle\alpha, \phi(x,y)v\rangle,\quad\forall \alpha\in V^*,x,y\in\g,v\in V.
\end{equation*}

\begin{lem}\label{dual-rep-3-Lie}
Let $(V;\rho)$ be a representation of a $3$-Lie algebra $(\g,[\cdot,\cdot,\cdot]_\g)$. Then $(V^*;\rho^*)$ is a representation of the $3$-Lie algebra $(\g,[\cdot,\cdot,\cdot]_\g)$, which is called the \emph{dual representation}.
\end{lem}

\begin{example}
Let $(\g,[\cdot,\cdot,\cdot]_\g)$ be a $3$-Lie algebra. The linear map $\ad:\wedge^{2}\g\longrightarrow\gl(\g)$ defined by
$$
\ad_{x,y}z=[x,y,z]_\g,\quad \forall x,y,z\in\g
$$
is a representation  of the $3$-Lie algebra $(\g,[\cdot,\cdot,\cdot]_\g)$, which we call the \emph{adjoint representation}. Furthermore, $(\g^*;\ad^*)$ is also a representation of $(\g,[\cdot,\cdot,\cdot]_\g)$.
\end{example}

An $n$-cochain on a 3-Lie algebra $(\g,[\cdot,\cdot,\cdot]_\g)$ with the coefficient in a representation $(V;\rho)$ is a linear map
$$f:\underbrace{\wedge^2\g \otimes {\dotsb} \otimes \wedge^2\g}_{n-1}\wedge\g \longrightarrow V.$$
Denote the space of $n$-cochains by $C^{n}_{3\mbox{-}\Lie}(\g;V),~(n\ge 1).$ The coboundary operator $\dM:C^{n}_{3\mbox{-}\Lie}(\g;V)\longrightarrow C^{n+1}_{3\mbox{-}\Lie}(\g;V)$ is given by
\begin{equation}\label{eq:drho}
\begin{split}
&(\dM f)(\frkX_1,\cdots ,\frkX_n,z)\\
={}& \sum_{1\leq j<k\leq n}(-1)^jf(\frkX_1,\cdots ,\hat{\frkX}_j,\cdots ,\frkX_{k-1},[\frkX_j,\frkX_k]_{\sfF},\frkX_{k+1},\cdots ,\frkX_{n},z)\\
&+ \sum_{j=1}^n(-1)^jf(\frkX_1,\cdots ,\hat{\frkX}_j,\cdots ,\frkX_{n},[\frkX_j,z]_\g)\\
&+ \sum_{j=1}^n(-1)^{j+1}\rho(\frkX_j)f(\frkX_1,\cdots ,\hat{\frkX}_j,\cdots ,\frkX_{n},z)\\
&+ (-1)^{n+1}\big(\rho(y_{n},z)f(\frkX_1,\cdots ,\frkX_{n-1},x_{n} ) +\rho(z,x_{n})f(\frkX_1,\cdots ,\frkX_{n-1},y_{n} ) \big),
\end{split}
\end{equation}
for all $\frkX_i=x_i\wedge y_i\in\wedge^2\g,~i=1,2,\cdots,n$ and $z\in\g.$ It was proved in \cite{Casas-Loday,Tcohomology} that $\dM\circ\dM=0$.

An element $f\in {C}^{n}_{3\mbox{-}\Lie}(\g;V)$ is called an $n$-cocycle if $\dM f=0$. It is called an $n$-coboundary if there exists some $g\in {C}^{n-1}_{3\mbox{-}\Lie}(\g;V)$ such that $f=\dM g$. Denote by $Z^n_{3\mbox{-}\Lie}(\g;V)$ and $B^n_{3\mbox{-}\Lie}(\g;V)$ the set of $n$-cocycles and the set of $n$-coboundaries respectively. Then we obtain the $n$-th cohomology group \begin{equation}\label{eq:cohomology}
H^n_{3\mbox{-}\Lie}(\g;V)=Z^n_{3\mbox{-}\Lie}(\g;V)/B^n_{3\mbox{-}\Lie}(\g;V).
\end{equation}

\begin{rmk}
The cohomology theory of 3-Lie algebras is closely related to the cohomology theory of Leibniz algebras. More precisely, let $(V;\rho)$ be a representation of
a $3$-Lie algebra $(\g,[\cdot,\cdot,\cdot]_\g)$. Then there is a representation of the Leibniz algebra $(\wedge^2\g,[\cdot,\cdot]_{\sfF})$ on the vector space $\Hom(\g,V)$ given by
\begin{align*}
\rho^L(x\wedge y)(\phi)z &= \rho(x,y)\phi(z)-\phi([x,y,z]_\g),\\
\rho^R(x\wedge y)(\phi)z &= \phi([x,y,z]_\g)-\rho(x,y)\phi(z)-\rho(y,z)\phi(x)-\rho(z,x)\phi(y),
\end{align*}
for all $\phi\in\Hom(\g,V), x,y,z\in \g$. Moreover, we define $\Phi:\Hom(\underbrace{\wedge^2\g \otimes {\dotsb} \otimes \wedge^2\g}_{n-1}\wedge\g, V)\lon\Hom(\otimes^{n-1}\wedge^2\g,\Hom(\g,V))$ by
$$
\Phi(f)(\frkX_1,\cdots ,\frkX_{n-1})x_n=f(\frkX_1,\cdots ,\frkX_{n-1},x_n),\quad\forall \frkX_1,\cdots ,\frkX_{n-1}\in\wedge^2\g,~x_n\in\g.
$$
Then $\Phi$ is a cochain map from the cochain complex $(C^n_{3\mbox{-}\Lie}(\g;V),\dM)$ to the cochain complex $(C^{n-1}_\Leib(\wedge^2\g;\Hom(\g,V)),\dM)$, i.e.~we have the following commutative diagram:
$$\xymatrix{
C^n_{3\mbox{-}\Lie}(\g;V)\ar[d]_{\dM}  \ar[r]^{\Phi\qquad} & C^{n-1}_\Leib(\wedge^2\g;\Hom(\g,V)) \ar[d]^{\dM}  \\
C^{n+1}_{3\mbox{-}\Lie}(\g;V)\ar[r]^{\Phi\qquad}   &C^{n}_\Leib(\wedge^2\g;\Hom(\g,V)).              }$$
Thus, it induces an isomorphism between cohomologies:
\begin{equation}
H^n_{3\mbox{-}\Lie}(\g;V)\cong H^{n-1}_{\Leib}\big(\wedge^2\g;\Hom(\g,V)\big),\quad\forall n=1,2,\cdots.
\end{equation}
See \cite{Casas-Loday,Figueroa} for more details about cohomology theory of 3-Lie algebras.
\end{rmk}

\section{The dgLa controlling deformations of algebraic structures}

In the section, we give the differential graded Lie algebras (dgLa for short) that control deformations of associative algebras, Lie algebras, pre-Lie algebras, Leibniz algebras and $3$-Lie algebras respectively.

Usually for an algebraic structure, Maurer-Cartan elements in a
suitable graded Lie algebra (gLa for short) are used to characterize realizations
of the algebraic structure on a space. For a given realization of
the algebraic structure, the corresponding Maurer-Cartan element
equips the graded Lie algebra with a differential. Then the
deformations of the given realization are characterized as the
Maurer-Cartan elements of the resulting differential graded Lie
algebra. Sometimes this differential graded Lie algebra is a commutator Lie algebra of a certain
\emph{differential graded associative algebra} resulting in an associative version of deformation theory.

We first recall a general notion and a basic fact~\cite{LV}.
\begin{defi}\label{def:mc}
Let $(\g=\oplus_{k\in\mathbb Z}\g^k,[\cdot,\cdot],d)$ be a differential graded Lie algebra.
A degree $1$ element $x\in\g^1$ is called a \emph{Maurer-Cartan element} of $\g$ if it satisfies the \emph{Maurer-Cartan equation}:
\begin{equation}
dx+\half[x,x]=0.
\label{eq:mce}
\end{equation}
\end{defi}
A graded Lie algebra is a differential graded Lie algebra with $d=0$.
One can similarly define the notion of a Maurer-Cartan element in a differential graded associative algebra.

\begin{defi}
Let $(\g=\oplus_{k\in\mathbb Z}\g^k,\cdot,d)$ be a differential graded associative algebra.
A degree $1$ element $x\in\g^1$ is called a \emph{Maurer-Cartan element} of $\g$ if it satisfies the \emph{Maurer-Cartan equation}:
\begin{equation}
dx+x\cdot x=0.
\end{equation}
\end{defi}

A graded associative algebra is a differential graded associative algebra with $d=0$. Then we have

\begin{thm}
Let $(\g=\oplus_{k\in\mathbb Z}\g^k,[\cdot,\cdot])$ be a graded Lie algebra and let $\mu\in \g^1$ be a Maurer-Cartan element. Then the map
$$ d_\mu: \g \longrightarrow \g, \ d_\mu(u):=[\mu, u], \quad \forall u\in \g,$$
is a differential on $\g$. For any $v\in \g^1$, the sum $\mu+v$ is a
Maurer-Cartan element of the graded Lie algebra $(\g,
[\cdot,\cdot])$ if and only if $v$ is a Maurer-Cartan element of the differential graded Lie algebra $(\g, [\cdot,\cdot], d_\mu)$. \label{pp:mce}
\end{thm}

A permutation $\sigma\in\mathbb S_n$ is called an $(i,n-i)$-shuffle if $\sigma(1)<\dotsb<\sigma(i)$ and $\sigma(i+1)<\dotsb <\sigma(n)$. If $i=0$ or $n$, we assume $\sigma=\id$. The set of all $(i,n-i)$-shuffles will be denoted by $\mathbb S_{(i,n-i)}$. The notion of an $(i_1,\cdots,i_k)$-shuffle and the set $\mathbb S_{(i_1,\cdots,i_k)}$ are defined analogously.

We also recall the notion of the suspension and desuspension operators. Let $V=\oplus_{i\in\mathbb Z} V^i$ be a graded vector space, we define the {\bf suspension operator} $\Sigma:V\mapsto \Sigma V$ by assigning  $V$ to the graded vector space $\Sigma V=\oplus_{i\in\mathbb Z}(\Sigma V)^i$ with $(\Sigma V)^i:=V^{i+1}$.
There is a natural degree $-1$ map $\Sigma:V\lon \Sigma V$ that is the identity map of the underlying vector space, sending $v\in V$ to its suspended copy $\Sigma v\in \Sigma V$. Likewise, the {\bf desuspension operator} $\Sigma^{-1}$ changes the grading of $V$ according to the rule $(\Sigma^{-1}V)^i:=V^{i-1}$. The  degree $1$ map $\Sigma^{-1}:V\lon \Sigma^{-1}V$ is defined in the obvious way.

\subsection{Associative algebras}

Let $\g$ be a vector space. Denote by $\frkC_\Ass^p(\g,\g)=\Hom(\otimes^{p+1}\g,\g)$ and set $\frkC^*_\Ass(\g,\g)=\oplus_{p\in\mathbb N}\frkC^p_\Ass(\g,\g).$ We assume that the degree of an element in $\frkC^p_\Ass(\g,\g)$ is $p$. Then we have
\begin{thm}{\rm(\cite{Ge0,Ge})}\label{gla-associative}
The graded vector space $\frkC^*_\Ass(\g,\g)$ equipped with the graded commutator bracket
\begin{equation}
[P,Q]_\G=P\circ Q-(-1)^{pq}Q\circ P,\quad \forall P\in \frkC^p_\Ass(\g,\g),~Q\in \frkC^q_\Ass(\g,\g),~ p,q\geq0,
\end{equation}
is a graded Lie algebra, where $P\circ Q\in \frkC_\Ass^{p+q}(\g,\g)$ is defined by
\begin{equation*}
(P\circ Q)(x_1,\cdots,x_{p+q+1}):=\sum_{i=1}^{p+1}(-1)^{(i-1)q}P(x_1,\cdots,x_{i-1},Q(x_{i},\cdots,x_{i+q}),x_{i+q+1},\cdots,x_{p+q+1}),
\end{equation*}
for all $x_1,\cdots,x_{p+q+1}\in \g$.

Furthermore, $(\g,\pi)$  is an associative algebra, where $\pi\in\Hom(\otimes^2\g,\g)$, if and only if $\pi$ is a Maurer-Cartan element of the graded Lie algebra $(\frkC^*_\Ass(\g,\g),[\cdot,\cdot]_\G)$, i.e.
$$[\pi,\pi]_\G=0.$$
\end{thm}
The bracket $[\cdot,\cdot]_\G$ is usually called the \emph{Gerstenhaber bracket}.
\begin{rmk}
Note that $(\frkC^*_\Ass(\g,\g),\circ)$ is a graded right symmetric algebra (pre-Lie algebra).
\end{rmk}
\begin{rmk}
In fact, the Gerstenhaber bracket is the commutator of  coderivations on the cofree conilpotent coassociative coalgebra $(\bar{T}^c(\Sigma\g),\bigtriangleup)$. See \cite{Stasheff} for more details.
\end{rmk}

Let $\pi:\g\otimes \g\lon \g$ be an associative algebra structure on a vector space $\g$. Since $\pi$ is a Maurer-Cartan element of the gLa $(\frkC^*_\Ass(\g,\g),[\cdot,\cdot]_\G)$ by Theorem~\ref{gla-associative}, it follows from the graded Jacobi identity that the map
\begin{equation}\label{eq:dT}
d_\pi:\Hom(\otimes^n\g,\g)\longrightarrow \Hom(\otimes^{n+1}\g,\g), \quad d_\pi f:=[\pi,f]_\G,
\end{equation}
is a graded derivation of the gLa $(\frkC^*_\Ass(\g,\g),[\cdot,\cdot]_\G)$ satisfying $d_\pi\circ d_\pi=0$.
Thus, we have
\begin{lem}\label{lem:dgla}
Let $\pi:\g\otimes \g\lon\g$ be an associative algebra structure on a vector space $\g$.  Then $(\frkC^*_\Ass(\g,\g),[\cdot,\cdot]_\G,d_\pi)$ is a differential graded Lie algebra.
\end{lem}

This dgLa can control deformations of the associative algebra structure $\pi$.

\begin{thm}\label{thm:deformation}
Let $\pi:\g\otimes \g\lon\g$ be an associative algebra structure on a vector space $\g$. Then for a linear map $\pi':\g\otimes \g\lon\g$, $\pi+\pi'$ is an associative algebra structure if and only if $\pi'$ is a Maurer-Cartan element of the dgLa $(\frkC^*_\Ass(\g,\g),[\cdot,\cdot]_\G,d_\pi)$, i.e.~$\pi'$ satisfies the Maurer-Cartan equation:
$$
d_\pi \pi'+\half[\pi',\pi']_\G=0.
$$
\end{thm}

We now construct a dga controlling deformations of morphisms of associative algebras. Let $\g$ and $\h$ be two associative algebras. Consider the graded vector space $\Sigma^{-1}\frkC^*_{\Ass}(\g,\h)$. Define a linear map $d:\Hom(\otimes^p\g,\h)\longrightarrow \Hom(\otimes^{p+1}\g,\h)$ by
\begin{eqnarray}
(d P) ( x_1,\cdots, x_{p+1} ):=(-1)^{p}\sum_{i=1}^{p}(-1)^{i}P(x_1,\cdots,x_{i-1},x_i\cdot x_{i+1},\cdots,x_{p+1}).
\end{eqnarray}
Also define a graded mutiplication
 $$\smallsmile: \Hom(\otimes^p\g,\h)\times \Hom(\otimes^q\g,\h)\longrightarrow \Hom(\otimes^{p+q}\g,\h)$$
by
\begin{eqnarray}
(P\smallsmile Q)(x_1,\cdots,x_{p+q}):=(-1)^{pq}P(x_1,\cdots,x_p)\cdot Q(x_{p+1},\cdots,x_{p+q})
\end{eqnarray}
for all  $P\in \frkC^{p-1}_\Ass(\g,\h)$,~$Q\in \frkC^{q-1}_\Ass(\g,\h)$,~$x_1,\cdots,x_{p+q}\in \g$.
\begin{thm}
Let $\g$ and $\h$ be two associative algebras. Then $(\Sigma^{-1}\frkC^*_{\Ass}(\g,\h),\smallsmile,d)$ is a differential graded associative algebra. Its Maurer-Cartan elements are precisely the associative algebra morphisms from $\g$ to $\h$.
\end{thm}

\begin{rmk}
In fact, if the two associative algebras $\g$ and $\h$ coincide, then there is a homotopy Gerstenhaber algebra structure on the graded vector space $\frkC^*_\Ass(\g,\g)$. See \cite{Ge5,Voronov} for more details about the notion of a homotopy Gerstenhaber algebra and Deligne conjecture.
\end{rmk}

In particular, a left module $V$ over an associative algebra $\g$ is an associative algebra morphism $\g \to \End(V)$, so we have

\begin{cor}
With the above notation, $V$ is a left module over $\g$ if and only if the corresponding map $\g \to \End(V)$ is a Maurer-Cartan element in $(\Sigma^{-1}\frkC^*_{\Ass}(\g,\End(V)),\smallsmile,d)$.
\end{cor}

Representations of associative algebras can also be characterized as Maurer-Cartan elements in a dgLa in the following way. Let $(\g,\pi)$ be an associative algebra, $\huaL,\huaR$ two linear maps from $\g$ to $\gl(V)$. Define $\bar{\pi},~\bar{\huaL},~\bar{\huaR}\in \Hom(\otimes^2(\g\oplus V),\g\oplus V)$ by
$$
\bar{\pi}(x+u,y+v)=\pi(x,y),~\bar{\huaL}(x+u,y+v)=\huaL_xv,~\bar{\huaR}(x+u,y+v)=\huaR_yu,~\forall x,y \in\g, ~u,v\in V.
$$
Then we have
\begin{prop}
With the above notations, $(V;\huaL,\huaR)$ is a representation of $(\g,\pi)$ if and only if $\bar{\huaL}+\bar{\huaR}$ is a Maurer-Cartan element of the dgLa  $\big(\frkC^*_\Ass(\g\oplus V,\g\oplus V),[\cdot,\cdot]_\G,d_{\bar{\pi}}=[\bar{\pi},\cdot]_\G\big)$.
\end{prop}
\begin{proof}
Since $(V;\huaL,\huaR)$ is a representation of $(\g,\pi)$ if and only if $\bar{\pi}+\bar{\huaL}+\bar{\huaR}$ is an associative algebra structure on the vector space $\g\oplus V$. Moreover, we know that $\bar{\pi}$ is an associative algebra structure on $\g\oplus V$.
It is straightforward to deduce that $[\bar{\pi}+\bar{\huaL}+\bar{\huaR},\bar{\pi}+\bar{\huaL}+\bar{\huaR}]_\G=0$ is equivalent to the following equality:
\begin{equation*}
2[\bar{\pi},\bar{\huaL}+\bar{\huaR}]_\G+[\bar{\huaL}+\bar{\huaR},\bar{\huaL}+\bar{\huaR}]_\G=0.
\end{equation*}
The proof is finished.
\end{proof}

At the end of this subsection, we give the relationship between the Hochschild coboundary operator and the differential $d_\pi$ defined by the Maurer-Cartan element. Assume that $(\g,\cdot)$ is an associative algebra and set
$\pi(x,y)=x\cdot y$. Consider the Hochschild coboundary operator
$\dM:\frkC^{n-1}_{\Ass}(\g;\g)\longrightarrow \frkC^{n}_{\Ass}(\g;\g)$ associated to the regular representation,
\begin{equation*}
\begin{split}
(\dM f)(x_1,\cdots,x_{n+1}) ={}& x_1\cdot f(x_2,\cdots,x_{n+1})+\sum_{i=1}^n(-1)^if(x_1,\cdots,x_{i-1},x_i\cdot x_{i+1},\cdots,x_{n+1})\\
&+ (-1)^{n+1}f(x_1,\cdots,x_n)\cdot x_{n+1}
\end{split}
\end{equation*}
for all $f\in \frkC^{n-1}_{\Ass}(\g;\g)$ and $x_1,x_2,\cdots,x_{n+1}\in \g$. Then we have
\begin{prop}
For all $f\in \frkC^{n-1}_{\Ass}(\g;\g)$, we have $$\dM f=(-1)^{n-1}d_\pi f=(-1)^{n-1}[\pi,f]_\G,~\forall n=1,2,\cdots.$$
\end{prop}
\begin{proof}
For all $x_1,x_2,\cdots,x_{n+1}\in \g$, we have
\begin{equation*}
\begin{split}
&(-1)^{n-1}[\pi,f]_\G(x_1,x_2,\cdots,x_{n+1})\\
={}& (-1)^{n-1}\pi(f(x_1,x_2,\cdots,x_{n}),x_{n+1})+(-1)^{n-1}(-1)^{n-1}\pi(x_1,f(x_2,\cdots,x_{n+1}))\\
&- (-1)^{n-1}(-1)^{n-1}\sum_{i=1}^{n}(-1)^{i-1}f(x_1,\cdots,x_{i-1},\pi(x_i,x_{i+1}),\cdots,x_{n+1})\\
={}& (\dM f)(x_1,\cdots,x_{n+1}),
\end{split}
\end{equation*}
which finishes the proof.
\end{proof}

\subsection{Lie algebras}

Let $\g$ be a vector space. Denote by $\frkC^p_\Lie(\g,\g)=\Hom(\wedge^{p+1}\g,\g)$ and set $\frkC^*_\Lie(\g,\g)=\oplus_{p\in\mathbb N}\frkC^p_\Lie(\g,\g).$ We assume that the degree of an element in $\frkC^p_\Lie(\g,\g)$ is $p$.

\begin{thm} For $P\in \frkC^p_\Lie(\g,\g), Q\in \frkC^q_\Lie(\g,\g)$, define the \emph{Nijenhuis-Richardson bracket} $[\cdot,\cdot]_{\NR}$ by
$$ [P,Q]_{\NR}:=P\circ Q- (-1)^{pq}Q\circ P,$$
where $P\circ Q\in \frkC^{p+q}_\Lie(\g,\g)$ is defined by
\begin{equation}
(P\circ Q)(x_1,\cdots,x_{p+q+1})
:=\smashoperator{\sum_{\sigma\in\mathbb S_{(q+1,p)}}} (-1)^\sigma P(Q(x_{\sigma(1)},\cdots,x_{\sigma(q+1)}),x_{\sigma(q+2)}, \cdots,x_{\sigma(p+q+1)}).
\label{eq:fgcirc}
\end{equation}
Then $(\frkC^*_\Lie(\g,\g),[\cdot,\cdot]_{\NR})$ is a graded Lie algebra.

Furthermore, $(\g,\pi)$ is a Lie algebra, where $\pi\in\Hom(\wedge^2\g,\g)$, if and only if $\pi$ is a Maurer-Cartan element of the graded Lie algebra $(\frkC^*_\Lie(\g,\g),[\cdot,\cdot]_{\NR})$, i.e.
$$[\pi,\pi]_\NR=0.$$
\end{thm}
Now we construct the differential graded Lie algebra that controls deformations of Lie algebras.
\begin{thm}\label{thm:deform-Lie}
Let $(\g,\pi)$ be a Lie algebra. Then $(\frkC^*_\Lie(\g,\g),[\cdot,\cdot]_{\NR},d_\pi)$ is a dgLa, where $d_\pi$ is defined by
\begin{equation}
d_\pi:=[\pi,\cdot]_{\NR}.
\end{equation}
Moreover, $\pi+\pi'$ is a Lie algebra structure on $\g$, where $\pi'\in \Hom(\wedge^2\g,\g)$, if and only if $\pi'$ is a Maurer-Cartan element of the dgLa $(\frkC^*_\Lie(\g,\g),[\cdot,\cdot]_{\NR},d_\pi)$.
\end{thm}

Representations of Lie algebras can be characterized as Maurer-Cartan elements in a dgLa. Let $(\g,\pi)$ be a Lie algebra, $\rho$ a linear map from $\g$ to $\gl(V)$. Define $\bar{\pi},\bar{\rho}\in \Hom(\wedge^2(\g\oplus V),\g\oplus V)$ by
$$
\bar{\pi}(x+u,y+v)=\pi(x,y),~\bar{\rho}(x+u,y+v)=\rho(x)v-\rho(y)u, \quad \forall x,y \in\g, ~u,v\in V.
$$
Then we have
\begin{prop}
With the above notations, $(V;\rho)$ is a representation of $(\g,\pi)$ if and only if $\bar{\rho}$ is a Maurer-Cartan element of the dgLa $\big(\frkC^*_\Lie(\g\oplus V,\g\oplus V),[\cdot,\cdot]_\NR,d_{\bar{\pi}}=[\bar{\pi},\cdot]_\NR\big)$, i.e.
$$
d_{\bar{\pi}}\bar{\rho}+\half[\bar{\rho},\bar{\rho}]_{\NR}=0.
$$
\end{prop}

At the end of this subsection, we give the relationship between the Chevalley-Eilenberg coboundary operator and the differential $d_\pi$ defined by the Maurer-Cartan element. Assume that $(\g,[\cdot,\cdot]_\g)$ is a Lie algebra and set
$\pi(x,y)=[x,y]_\g$. Consider the Chevalley-Eilenberg coboundary operator
$\dM:\frkC^{n-1}_{\Lie}(\g;\g)\longrightarrow \frkC^{n}_{\Lie}(\g;\g)$ associated to the adjoint representation, then we have
\begin{prop}
For all $f\in \frkC^{n-1}_{\Lie}(\g;\g)$, we have $$\dM f=(-1)^{n-1}d_\pi f=(-1)^{n-1}[\pi,f]_{\NR},~\forall n=1,2,\cdots.$$
\end{prop}

\subsection{pre-Lie algebras}

Let $\g$ be a vector space. Denote by $\frkC^p_\preLie(\g,\g)=\Hom(\wedge^{p}\g\otimes \g,\g)$ and set $\frkC^*_\preLie(\g,\g)=\oplus_{p\in\mathbb N}\frkC^p_\preLie(\g,\g).$ We assume that the degree of an element in $\frkC^p_\preLie(\g,\g)$ is $p$.

For $P\in\frkC^p_\preLie(\g,\g)$ and $Q\in\frkC^q_\preLie(\g,\g)$, define  $P\circ Q\in\frkC^{p+q}_\preLie(\g,\g)$ by
\begin{equation}\mlabel{eq:pLbrac}
\begin{split}
&(P\circ Q)(x_1,\cdots,x_{p+q+1})\\
:={}& \sum_{\sigma\in\mathbb S_{(q,1,p-1)}}(-1)^{\sigma}P(Q(x_{\sigma(1)},\cdots,x_{\sigma(q+1)}),x_{\sigma(q+2)},\cdots,x_{\sigma(p+q)},x_{p+q+1})\\
&+ (-1)^{pq}\sum_{\sigma\in\mathbb S_{(p,q)}}(-1)^{\sigma}P(x_{\sigma(1)},\cdots,x_{\sigma(p)},Q(x_{\sigma(p+1)},\cdots,x_{\sigma(p+q)},x_{p+q+1})).
\end{split}
\end{equation}

\begin{thm}
The graded vector space
$\frkC^*_\preLie(\g,\g)$ equipped with the \emph{Matsushima-Nijenhuis bracket} $[\cdot,\cdot]_{\MN}$ given by
\begin{equation}
[P,Q]_{\MN}:=P\circ Q-(-1)^{pq}Q\circ P,
\quad \forall P\in\frkC^p_\preLie(\g,\g),~Q\in\frkC^q_\preLie(\g,\g),
\mlabel{eq:glapL}
\end{equation}
is a graded Lie algebra.
\end{thm}
See \cite{CL,Nijenhuis,WBLS} for more details about the Matsushima-Nijenhuis bracket.

\begin{rmk}
For $\pi\in\Hom(\g\otimes \g,\g)$, we have
\begin{equation*}
\begin{split}
[\pi,\pi]_{\MN}(x,y,z)&= 2(\pi\circ\pi)(x,y,z)\\
&= 2\big(\pi(\pi(x,y),z)-\pi(\pi(y,x),z)-\pi(x,\pi(y,z))+\pi(y,\pi(x,z))\big).
\end{split}
\end{equation*}
Thus, $\pi$ defines a pre-Lie algebra structure on $\g$ if and only if $[\pi,\pi]_{\MN}=0,$ that is, $\pi$ is a Maurer-Cartan element  of the graded Lie algebra $(\frkC^*_\preLie(\g,\g),[\cdot,\cdot]_{\MN})$.
\label{rk:plmc}
\end{rmk}

Now we construct the dgLa that controls deformations of pre-Lie algebras.
\begin{thm}
Let $(\g,\pi)$ be a pre-Lie algebra. Then $(\frkC^*_\preLie(\g,\g),[\cdot,\cdot]_{\MN},d_\pi)$ is a dgLa, where $d_\pi$ is defined by
\begin{equation}
d_\pi:=[\pi,\cdot]_{\MN}.
\end{equation}
Moreover, $\pi+\pi'$ is a pre-Lie algebra structure on $\g$, where $\pi'\in \Hom(\g\otimes\g,\g)$, if and only if $\pi'$ is a Maurer-Cartan element of the dgLa $(\frkC^*_\preLie(\g,\g),[\cdot,\cdot]_{\MN},d_\pi)$.
\end{thm}

Representations of pre-Lie algebras can be characterized by a dgLa. Let $(\g,\pi)$ be a pre-Lie algebra, $\rho,\mu$ two linear maps from $\g$ to $\gl(V)$. Define $\bar{\pi},\bar{\rho},\bar{\mu}\in \Hom(\otimes^2(\g\oplus V),\g\oplus V)$ by
$$
\bar{\pi}(x+u,y+v)=\pi(x,y),~\bar{\rho}(x+u,y+v)=\rho(x)v,~\bar{\mu}(x+u,y+v)=\mu(y)u, \quad \forall x,y \in\g, ~u,v\in V.
$$
Then we have
\begin{prop}
With the above notations, $(V;\rho,\mu)$ is a representation of $(\g,\pi)$ if and only if $\bar{\rho}+\bar{\mu}$ is a Maurer-Cartan element of the dgLa $\big(\frkC^*_\preLie(\g\oplus V,\g\oplus V),[\cdot,\cdot]_\NR,d_{\bar{\pi}}=[\bar{\pi},\cdot]_\MN\big)$, i.e.
$$
d_{\bar{\pi}}(\bar{\rho}+\bar{\mu})+\half[\bar{\rho}+\bar{\mu},\bar{\rho}+\bar{\mu}]_{\MN}=0.
$$
\end{prop}

At the end of this subsection, we give the relationship between the coboundary operator  of a pre-Lie algebra and the differential $d_\pi$ defined by the Maurer-Cartan element. Assume that $(\g,\cdot_\g)$ is a pre-Lie algebra and set
$\pi(x,y)=x\cdot_\g y$. Consider the operator
$\dM:\frkC^{n-1}_{\preLie}(\g;\g)\longrightarrow \frkC^{n}_{\preLie}(\g;\g)$ associated to the regular representation, then we have
\begin{prop}
For all $f\in \frkC^{n-1}_{\preLie}(\g;\g)$, we have
$$\dM f=(-1)^{n-1}d_\pi f=(-1)^{n-1}[\pi,f]_{\MN},~\forall n=1,2,\cdots.$$
\end{prop}

\subsection{Leibniz algebras}

Let $\g$ be a vector space. Denote by $\frkC^p_\Leib(\g,\g)=\Hom(\otimes^{p+1}\g,\g)$ and set $\frkC^*_\Leib(\g,\g)=\oplus_{p\in\mathbb N}\frkC^p_\Leib(\g,\g).$ We assume that the degree of an element in $\frkC^p_\Leib(\g,\g)$ is $p$.

\begin{thm} The graded vector space $\frkC^*_\Leib(\g,\g)$ equipped with the \emph{Balavoine bracket}
\begin{equation}\label{leibniz-bracket}
[P,Q]_\B=P\circ Q-(-1)^{pq}Q\circ P, \quad \forall P\in \frkC_\Leib^{p}(\g,\g),Q\in \frkC_\Leib^{q}(\g,\g),
\end{equation}
is a graded Lie algebra, where $P\circ Q\in \frkC_\Leib^{p+q}(\g,\g)$ is defined by
\begin{equation}
P\circ Q=\sum_{k=1}^{p+1}(-1)^{(k-1)q}P\circ_k Q,
\end{equation}
and $P\circ_k Q$ is defined by
\begin{equation*}
\begin{split}
&(P\circ_kQ)(x_1,\cdots,x_{p+q+1})\\
={}& \sum_{\sigma\in\mathbb S_{(k-1,q)}}(-1)^{\sigma}P(x_{\sigma(1)},\cdots,x_{\sigma(k-1)},Q(x_{\sigma(k)},\cdots,x_{\sigma(k+q-1)},x_{k+q}),x_{k+q+1},\cdots,x_{p+q+1}).
\end{split}
\end{equation*}
\end{thm}

\begin{rmk}
We have the compact formula of the Balavoine bracket. Let $\g$ be a vector space. We consider the graded vector space $\frkC^*_\Leib(\g,\g)=\oplus_{p\in\mathbb N}\frkC^p_\Leib(\g,\g)=\oplus_{p\in\mathbb N}\Hom(\otimes^p\g,\gl(\g))$. The Balavoine bracket is given as following:
\small{
\begin{equation*}
\begin{split}
&([P,Q]_\B)(x_1,\cdots,x_{p+q})\\
={}& \sum_{k=1}^{p}(-1)^{(k-1)q}\smashoperator{\sum_{\sigma\in\mathbb S_{(k-1,q)}}}
(-1)^{\sigma}P(x_{\sigma(1)},\cdots,x_{\sigma(k-1)},Q(x_{\sigma(k)},\cdots,x_{\sigma(k+q-1)})x_{k+q},x_{k+q+1},\cdots,x_{p+q})\\
&- (1)^{pq}\sum_{k=1}^{q}(-1)^{(k-1)p}\smashoperator{\sum_{\sigma\in\mathbb S_{(k-1,p)}}}
(-1)^{\sigma}Q(x_{\sigma(1)},\cdots,x_{\sigma(k-1)},P(x_{\sigma(k)},\cdots,x_{\sigma(k+p-1)})x_{k+p},x_{k+p+1},\cdots,x_{p+q})\\
&+ (-1)^{pq}\smashoperator{\sum_{\sigma\in\mathbb S_{(p,q)}}}
(-1)^{\sigma}[P(x_{\sigma(1)},\cdots,x_{\sigma(p)}), Q(x_{\sigma(p+1)},\cdots,x_{\sigma(p+q)})].
\end{split}
\end{equation*}
}
\end{rmk}

See \cite{Balavoine-1,Fialowski,Sheng-Tang} for more details about the Balavoine bracket. In particular, for $\pi\in \frkC_\Leib^{1}(\g,\g)$, we have
\begin{equation*}
[\pi,\pi]_\B(x,y,z)=2(\pi\circ\pi)(x,y,z)=2\big(\pi(\pi(x,y),z)-\pi(x,\pi(y,z))
+\pi(y,\pi(x,z))\big).
\end{equation*}

Thus, $\pi$ defines a Leibniz algebra structure if and only if $[\pi,\pi]_{\B}=0$, i.e.~$\pi$ is a Maurer-Cartan element of the gLa $(\frkC^*_\Leib(\g,\g),[\cdot,\cdot]_\B)$.

Now we construct the dgLa that controls deformations of Leibniz algebras.
\begin{thm}
Let $(\g,\pi)$ be a Leibniz algebra. Then $(\frkC^*_\Leib(\g,\g),[\cdot,\cdot]_{\B},d_\pi)$ is a dgLa, where $d_\pi$ is defined by
\begin{equation}
d_\pi:=[\pi,\cdot]_{\B}.
\end{equation}
Moreover, $\pi+\pi'$ is a Leibniz algebra structure on $\g$, where $\pi'\in \Hom(\g\otimes\g,\g)$, if and only if $\pi'$ is a Maurer-Cartan element of the dgLa $(\frkC^*_\Leib(\g,\g),[\cdot,\cdot]_{\B},d_\pi)$.
\end{thm}

Representations of Leibniz algebras can be characterized by a dgLa. Let $(\g,\pi)$ be a Leibniz algebra, $\rho^L,\rho^R$ two linear maps from $\g$ to $\gl(V)$. Define $\bar{\pi},\bar{\rho}^L,\bar{\rho}^R\in \Hom(\otimes^2(\g\oplus V),\g\oplus V)$ by
$$
\bar{\pi}(x+u,y+v)=\pi(x,y),\quad\bar{\rho}^L(x+u,y+v)=\rho^L(x)v,\quad \bar{\rho}^R(x+u,y+v)=\rho^R(y)u,
$$
for all $x,y \in\g, ~u,v\in V.$ Then we have
\begin{prop}
With the above notations, $(V;\rho^L,\rho^R)$ is a representation of $(\g,\pi)$ if and only if $\bar{\rho}^L+\bar{\rho}^R$ is a Maurer-Cartan element of the dgLa $\big(\frkC^*_\Leib(\g\oplus V,\g\oplus V),[\cdot,\cdot]_\B,d_{\bar{\pi}}=[\bar{\pi},\cdot]_\B\big)$.
\end{prop}

At the end of this subsection, we give the relationship between the coboundary operator of a Leibniz algebra and the differential $d_\pi$ defined by the Maurer-Cartan element. Assume that $(\g,[\cdot,\cdot]_\g)$ is a Leibniz algebra and set
$\pi(x,y)=[x,y]_\g$. Consider the operator
$\dM:\frkC^{n-1}_{\Leib}(\g;\g)\longrightarrow \frkC^{n}_{\Leib}(\g;\g)$ associated to the regular representation, then we have
\begin{prop}
For all $f\in \frkC^{n-1}_{\Leib}(\g;\g)$, we have $$\dM f=(-1)^{n-1}d_\pi f=(-1)^{n-1}[\pi,f]_{\B},~\forall n=1,2,\cdots.$$
\end{prop}

\subsection{$3$-Lie algebras}

In \cite{NR bracket of n-Lie}, the author constructed a graded Lie algebra structure by which one can describe an $n$-Leibniz algebra structure as a Maurer-Cartan element. Here, we give the precise formulas for the 3-Lie algebra case.

Set $\frkC^p_{3\mbox{-}\Lie}(\g,\g)=\Hom(\underbrace{\wedge^2\g\otimes{\dotsb}\otimes\wedge^2\g}_{p}\wedge\g,\g)$ and $\frkC^*_{3\mbox{-}\Lie}(\g,\g)=\oplus_{p}\frkC^{p}_{3\mbox{-}\Lie}(\g,\g)$.

Let $P\in \frkC^p_{3\mbox{-}\Lie}(\g,\g),~Q\in \frkC^q_{3\mbox{-}\Lie}(\g,\g),\, p,q\geq 0$. Let $\frkX_i=x_i\wedge y_i\in \wedge^2\g$ for $i=1,2,\cdots,p+q$ and $x_i,y_i\in\g$.

\begin{thm}{\rm (\cite{NR bracket of n-Lie,Song-Tang-Makhlouf})}\label{thm:gradelie}
The graded vector space $\frkC_{3\mbox{-}\Lie}(\g,\g)$ equipped with the graded commutator bracket
\begin{equation}
[P,Q]_{3\mbox{-}\Lie}= P\circ Q-(-1)^{pq}Q\circ P,
\end{equation}
is a graded Lie algebra, where $P\circ Q\in \frkC^{p+q}_{3\mbox{-}\Lie}(\g,\g)$ is defined by
{\footnotesize
\begin{equation*}
\begin{split}
&(P\circ Q)(\frkX_1,\cdots,\frkX_{p+q},x)\\
={}& \sum_{k=1}^{p}(-1)^{(k-1)q}\smashoperator{\sum_{\sigma\in \mathbb S_{(k-1,q)}}}
(-1)^{\sigma}P\left(\frkX_{\sigma(1)},\cdots,\frkX_{\sigma(k-1)},Q(\frkX_{\sigma(k)},\cdots,\frkX_{\sigma(k+q-1)},x_{k+q})\wedge y_{k+q},\frkX_{k+q+1},\cdots,\frkX_{p+q},x\right)\\
&+ \sum_{k=1}^{p}(-1)^{(k-1)q}\smashoperator{\sum_{\sigma\in \mathbb S_{(k-1,q)}}}
(-1)^{\sigma}P\left(\frkX_{\sigma(1)},\cdots,\frkX_{\sigma(k-1)},x_{k+q}\wedge Q(\frkX_{\sigma(k)},\cdots,\frkX_{\sigma(k+q-1)},y_{k+q}) ,\frkX_{k+q+1},\cdots,\frkX_{p+q},x\right)\\
&+ \smashoperator{\sum_{\sigma\in \mathbb S_{(p,q)}}}
(-1)^{pq}(-1)^{\sigma}P\left(\frkX_{\sigma(1)},\cdots,\frkX_{\sigma(p)},Q(\frkX_{\sigma(p+1)},\cdots,\frkX_{\sigma(p+q-1)}, \frkX_{\sigma(p+q)},x)\right).
\end{split}
\end{equation*}
}
Moreover,  $\pi:\wedge^3\g\longrightarrow\g$ is a $3$-Lie bracket if and only if $[\pi,\pi]_{3\mbox{-}\Lie}=0$, i.e.~$\pi$ is a Maurer-Cartan element of the graded Lie algebra $(\frkC^*_{3\mbox{-}\Lie}(\g,\g),[\cdot,\cdot]_{3\mbox{-}\Lie})$.
\end{thm}

Define a graded linear map $\Psi:\frkC^*_{3\mbox{-}\Lie}(\g,\g)\lon\frkC^*_\Leib(\wedge^2\g,\wedge^2\g)$ of degree $0$ by
\begin{equation*}
\begin{split}
&\Psi(P)(\frkX_1,\cdots,\frkX_{p},x_{p+1}\wedge y_{p+1})\\
={}& P(\frkX_1,\cdots,\frkX_{p},x_{p+1})\wedge y_{p+1}+x_{p+1}\wedge P(\frkX_1,\cdots,\frkX_{p},y_{p+1}),
\quad \forall P\in\frkC^{p}_{3\mbox{-}\Lie}(\g,\g),\quad p\geq 0.
\end{split}
\end{equation*}

We have
\begin{thm}
With the above notations, Then $\Phi$ is a
homomorphism of graded Lie algebras from $(\frkC^*_{3\mbox{-}\Lie}(\g,\g),[\cdot,\cdot]_{3\mbox{-}\Lie})$ to $\big(\frkC^*_\Leib(\wedge^2\g,\wedge^2\g),[\cdot,\cdot]_{\B}\big)$.
\end{thm}

\begin{cor}\label{rmk:O-pre}
Let $(\g,[\cdot,\cdot,\cdot]_\g)$ is a $3$-Lie algebra and set $\pi(x,y,z)=[x,y,z]_\g$. Then $\Psi(\pi)$ is a Maurer-Cartan element of the graded Lie algebra $\big(\frkC^*_\Leib(\wedge^2\g,\wedge^2\g),[\cdot,\cdot]_{\B}\big)$, which is exactly the Leibniz algebra structure on the fundamental objects $\wedge^2\g$ given by \eqref{eq:bracketfunda}.
\end{cor}

Now we construct the dgLa that controls deformations of $3$-Lie algebras.
\begin{thm}
Let $(\g,\pi)$ be a $3$-Lie algebra. Then $(\frkC^*_{3\mbox{-}\Lie}(\g,\g),[\cdot,\cdot]_{3\mbox{-}\Lie},d_\pi)$ is a dgLa, where $d_\pi$ is defined by
\begin{equation}
d_\pi:=[\pi,\cdot]_{3\mbox{-}\Lie}.
\end{equation}
Moreover, $\pi+\pi'$ is a pre-Lie algebra structure on $\g$, where $\pi'\in \Hom(\wedge^3\g,\g)$, if and only if $\pi'$ is a Maurer-Cartan element of the dgLa $(\frkC^*_{3\mbox{-}\Lie}(\g,\g),[\cdot,\cdot]_{3\mbox{-}\Lie},d_\pi)$.
\end{thm}

Representations of $3$-Lie algebras can be characterized by a dgLa. Let $(\g,\pi)$ be a $3$-Lie algebra, $\rho$ a linear map from $\wedge^2\g$ to $\gl(V)$. Define $\bar{\pi},\bar{\rho}\in \Hom(\wedge^3(\g\oplus V),\g\oplus V)$ by
$$
\bar{\pi}(x+u,y+v,z+w)=\pi(x,y,z),~\bar{\rho}(x+u,y+v,z+w)=\rho(x,y)w+\rho(y,z)u+\rho(z,x)v,
$$
for all $x,y,z\in\g, ~u,v,w\in V$. Then we have
\begin{prop}
With the above notations, $(V;\rho)$ is a representation of $(\g,\pi)$ if and only if $\bar{\rho}$ is a Maurer-Cartan element of the dgLa $\big(\frkC^*_{3\mbox{-}\Lie}(\g\oplus V,\g\oplus V),[\cdot,\cdot]_{3\mbox{-}\Lie},d_{\bar{\pi}}=[\bar{\pi},\cdot]_{3\mbox{-}\Lie}\big)$.
\end{prop}

At the end of this subsection, we give the relationship between the coboundary operator of a $3$-Lie algebra and the differential $d_\pi$ defined by the Maurer-Cartan element. Assume that $(\g,[\cdot,\cdot,\cdot]_\g)$ is a $3$-Lie algebra  and set $\pi(x,y,z)=[x,y,z]_\g$. Consider the operator
$\dM:\frkC^{n-1}_{3\mbox{-}\Lie}(\g;\g)\longrightarrow \frkC^{n}_{3\mbox{-}\Lie}(\g;\g)$ associated to the adjoint representation, then we have
\begin{prop}
For all $f\in \frkC^{n-1}_{3\mbox{-}\Lie}(\g;\g)$, we have $$\dM f=(-1)^{n-1}d_\pi f=(-1)^{n-1}[\pi,f]_{3\mbox{-}\Lie},~\forall n=1,2,\cdots.$$
\end{prop}

\section{Deformations of infinity-structures}
Most naturally occurring algebraic structures usually admit ``strong homotopy versions''; they are obtained by allowing their defining identities  to hold up to ``higher homotopies''. A systematic approach to strong homotopy algebras relies on the theory of \emph{operads}, cf. \cite{LV}, which is outside the scope of this paper.  Instead we will restrict ourselves with considering two representative examples: $A_\infty$ and $L_\infty$-algebras (also known as strong homotopy associative and strong homotopy Lie algebras) and their morphisms.

$A_\infty$ and $L_\infty$-algebras are traditionally defined as graded vector spaces $\g$ supplied with collections of multilinear maps $\g^{\otimes n}\to \g, n=1,2,\ldots$ satisfying various constraints, cf. \cite{fp02,Keller, ls93}. We choose a different, albeit equivalent, approach, which posits that an $A_\infty$ or $L_\infty$-structure is determined by a Maurer-Cartan element in a suitable graded Lie algebra. This has the advantage of being concise, conceptual, and avoids many cumbersome sign issues. This definition is adopted in e.g. \cite{HL} where it is also explained how it is equivalent to the one given by multilinear maps.

We need a few preliminaries regarding completed tensor and symmetric algebras. Let $\g$ be a graded vector space and consider the tensor algebra $T\g^*:=\bigoplus_{n=0}^\infty (\g^*)^{\otimes n}$ on its linear dual $\g^*$. If $\h\subset \g$ is a \emph{finite-dimensional} subspace of $\g$, then restriction determines an algebra map
$$T\g^*\to (T\h)^*\cong \prod_{n=0}^\infty (\h^*)^{\otimes n}.$$
Then set $\hat{T}\h^*:=\varprojlim \prod_{n=0}^\infty (\h^*)^{\otimes n}$ where the inverse limit is taken over all finite-dimensional subspaces $\h$ of $\g$. Note that if $\g$ is itself finite-dimensional, then $\hat{T}\g^*:=\prod_{n=0}^\infty (\g^*)^{\otimes n}$. We will call $\hat{T}\g^*$ the \emph{completed tensor algebra} on $\g^*$.

It is clear that $T\g^*\cong \prod_{n=0}^\infty (\g^*)^{\hat{\otimes} n}$ where $(\g^*)^{\hat{\otimes} n}$ is a certain completion of $(\g^*)^{{\otimes} n}$; in fact $(\g^*)^{\hat{\otimes} n}\cong (\g^{\otimes n})^*$. The symmetric group $\mathbb S_n$ acts on $(\g^*)^{{\hat{\otimes}} n}$ by permuting the tensor factors and we define the \emph{completed symmetric algebra} on $\g^*$ as $\hat{S}\g^*:=\prod_{n=0}^\infty (\g^*)^{\hat{\otimes n} }_{\mathbb S_n}$.

Note that the vector space $\g^*$ is an inverse limit of its finite-dimensional quotients (since $\g$ is a union of its finite-dimensional subspaces); it is, thus, a \emph{topological} vector space (with the inverse limit topology). Similarly, $\hat{T}\g^*$ and $\hat{S}\g^*$ are \emph{topological} algebras (as inverse limits of discrete algebras). When considering maps involving $\g^*$, $\hat{T}\g^*$ and $\hat{S}\g^*$ or derivations of $T\hat{\g}^*$ and $\hat{S}\g^*$, we will always assume them to be continuous without mentioning that explicitly.

The completed symmetric and tensor algebras possess certain universal properties, similar to the uncompleted versions. Namely, for another vector space $\h$ it holds that \begin{itemize}
	\item The set of algebra maps $\hat{T}\g^*\to \hat{T}\h^*$ is in 1-1 correspondence with linear maps
	$$\g^*\to \hat{T}_+\h^*:=\prod_{n=1}^\infty (\h^*)^{\hat{\otimes n}}\subset\hat{T}\h^*$$ (so that any such map is determined by its restriction to $\g^*$ and the latter lands in the maximal ideal $\hat{T}_+\h^*$ of $\hat{T}\h^*$) and the set $\Der (\hat{T}\g^*)$ is in 1-1 correspondence with maps $\g^*\to \hat{T}\g^*$ (so that any derivation is determined by its restriction onto $\g^*$).
	\item The set of algebra maps $\hat{S}\g^*\to \hat{S}\h^*$ is in 1-1 correspondence with linear maps $$\g^*\to \hat{S}_+\h^*:=\prod_{n=1}^\infty (\h^*)^{\hat{\otimes n}}_{\mathbb S_n}$$ (so that any such map is determined by its restriction to $\g^*$ and the latter lands in the maximal ideal $\hat{S}_+\h^*$ of $\hat{S}\g^*$) and the set $\Der (\hat{S}\g^*)$ is in 1-1 correspondence with maps $\g^*\to \hat{S}\g^*$ (so that any derivation is determined by its restriction onto $\g^*$).	
\end{itemize}
If a map $\g^*\to \hat{T}\g^*$ has its image in the maximal ideal $\hat{T}_+\g^*$ of $\hat{T}\g^*$, we say that the corresponding derivation of $\Der \hat{T}\g^*$ \emph{vanishes at zero}. Similarly if a map $\g^*\to \hat{S}\g^*$ has its image in the maximal ideal $\hat{S}_+\g^*$ of $\hat{S}\g^*$, we say that the corresponding derivation of $\Der \hat{T}\g^*$ vanishes at zero. Derivations vanishing at zero form Lie subalgebras in $\Der \hat{T}\g^*$ and $\Der \hat{S}\g^*$ that will be denoted by $\Der_0 \hat{T}\g^*$ and $\Der_0 \hat{S}\g^*$ respectively.
\begin{rem}
Let $\g$ be a non-unital differential graded associative algebra. Then the differential $\g\to\g$ gives rise to a degree $1$ map $m_1: \Sigma^{-1}\g^*\to \Sigma^{-1}\g^*$ and the multiplication map $\g\otimes\g\to \g$ similarly gives a degree $1$ map $m_2:\Sigma^{-1}\g^*\to \Sigma^{-1}\g^*\hat{\otimes}\Sigma^{-1}\g^*$. The resulting map $m_1+m_2:\Sigma^{-1}\g^*\to \Sigma^{-1}\g^*\oplus \Sigma^{-1}\g^*\hat{\otimes}\Sigma^{-1}\g^*$ corresponds to a derivation $m$ of $\hat{T}V^*$ and one can show that $m^2=0$. Thus, $A_\infty$-algebras generalize (non-unital) associative algebras. Arguing similarly, we conclude that $L_\infty$-algebras are a generalization of differential graded Lie algebras.
\end{rem}
 We can now give compact definitions of $A_\infty$ and $L_\infty$-algebras.
\begin{defi} \label{def_infstr}
	Let $\g$ be a  graded vector space:
	\begin{enumerate}
		\item
		An $A_\infty$-structure on $\g$ is a derivation $ m\in\Der_0(\hat{T}\Sigma^{-1}\g^*) $
		of degree $1$ such that $m^2=0$.
		\item
		An $L_\infty$-structure on $\g$ is a derivation
		$ m\in\Der_0(\hat{S}\Sigma^{-1}\g^*)$
		of degree $1$ such that $m^2=0$.
		\end{enumerate}
\end{defi}
In other words, an $A_\infty$-structure on $\g$ is a Maurer-Cartan element in $\Der_0(\hat{T}\g^*)$ and an $L_\infty$-structure
on $\g$ is a Maurer-Cartan element in $\Der_0(\hat{S}\g^*)$.

\emptycomment{\begin{rmk}
In fact, we can deduced the concrete definitions of $A_\infty$ and $L_\infty$-algebras from the above definition:
\begin{enumerate}
\item
An $A_{\infty}$-algebra is a $\mathbb Z$-graded vector space $ \g=\oplus_{i\in \mathbb Z} \g^i$ equipped with a collection $(n\ge 1)$ of linear maps $m_n:\otimes^n\g\lon\g$ of degree $n-2$ with the property that, for any homogeneous elements $x_1,\cdots,x_n\in \g$, we have the {\bf Stasheff identities} for $n\ge 1$
\begin{eqnarray*}
\sum_{i=1}^{n}\sum_{r=0}^{n-i}(-1)^{r+i(n-r-i)+i(x_1+\cdots+x_r)}m_{n+1-i}(x_1,\cdots,x_r,m_{i}(x_{i+1},\cdots,x_{i+r}),x_{i+r+1},\cdots,x_n)=0.
\end{eqnarray*}

\item
An $L_{\infty}$-algebra is a $\mathbb Z$-graded vector space $ \g=\oplus_{i\in \mathbb Z} \g^i$ equipped with a collection $(n\ge 1)$ of linear maps $l_n:\otimes^n\g\lon\g$ of degree $n-2$ with the property that, for any homogeneous elements $x_1,\cdots,x_n\in \g$, we have
\begin{itemize}\item[\rm(i)]
{\bf (graded antisymmetry)} for every $\sigma\in\mathbb S_{n}$,
\begin{eqnarray*}
l_n(x_{\sigma(1)},\cdots,x_{\sigma(n)})=(-1)^{\sigma}\varepsilon(\sigma)l_n(x_1,\cdots,x_n),
\end{eqnarray*}
\item[\rm(ii)] {\bf (generalized Jacobi identities)} for all $n\ge 1$,
\begin{eqnarray*}\label{sh-Lie}
\sum_{i=1}^{n}\sum_{\sigma\in \mathbb S_{(i,n-i)} }(-1)^{i(n-i)}(-1)^{\sigma}\varepsilon(\sigma)l_{n-i+1}(l_i(x_{\sigma(1)},\cdots,x_{\sigma(i)}),x_{\sigma(i+1)},\cdots,x_{\sigma(n)})=0.
\end{eqnarray*}
\end{itemize}
\end{enumerate}
\end{rmk}
}

The notions of $A_\infty$ and $L_\infty$-morphisms are defined similarly.
\begin{defi} \label{def_infmor}
	Let $\g$ and $\h$ be graded vector spaces:
	\begin{enumerate}
		\item
		Let $m$ and $m'$ be $A_\infty$-structures on $\g$ and $\h$ respectively. An $A_\infty$-morphism from $\g$ to $\h$ is a an algebra homomorphism
		\[ \phi:\hat{T}\Sigma^{-1}\h^* \to \hat{T}\Sigma^{-1}\g^* \]
		of degree zero such that $\phi \circ m'=m \circ \phi$.
		\item
		Let $m$ and $m'$ be $L_\infty$-structures on $\g$ and $\h$ respectively. An $L_\infty$-morphism from $\g$ to $\h$ is an algebra homomorphism
		\[ \phi:\hat{S}\Sigma^{-1}\h^* \to \hat{S}\Sigma^{-1}\g^* \]
		of degree zero such that $\phi \circ m'=m \circ \phi$.
	\end{enumerate}
\end{defi}
We have the following infinity-analogue of Theorem \ref{thm:deformation} and Theorem \ref{thm:deform-Lie}.
\begin{thm}Let $\g$ be a graded vector space: \begin{enumerate}
		\item
	 Let $m\in \Der_0(\hat{T}\Sigma^{-1}\g^*)$ be an $A_\infty$-structure on $\g$. Then
for an element $m^\prime\in \Der_0(\hat{T}\Sigma^{-1}\g^*)$ of degree $1$, the element $m+m^\prime\in \Der_0(\hat{T}\Sigma^{-1}\g^*)$ is an $A_\infty$-structure on $\g$ if and only if $m^\prime$ is a Maurer-Cartan element in the
differential graded Lie algebra $(\Der_0(\hat{T}\Sigma^{-1}\g^*), d_m)$.	
\item Let $m\in \Der_0(\hat{S}\Sigma^{-1}\g^*)$ be an $L_\infty$-structure on  $\g$. Then
for an element $m^\prime\in \Der_0(\hat{S}\Sigma^{-1}\g^*)$ of degree $1$, the element $m+m^\prime\in \Der_0(\hat{S}\Sigma^{-1}\g^*)$ is an $L_\infty$-structure on $\g$ if and only if $m^\prime$ is a Maurer-Cartan element in the
differential graded Lie algebra $(\Der_0(\hat{S}\Sigma^{-1}\g^*), d_m)$.
\end{enumerate}
	\end{thm}
Let $(\g, m_\g)$ be an $A_\infty$-algebra and $\h$ be a (possibly non-unital) differential graded associative algebra \emph{viewed as an
$A_\infty$-algebra}. It follows from Definition \ref{def_infmor} that an $A_\infty$-morphism $\g\to \h$ is just a Maurer-Cartan element in the differential graded associative algebra $\hat{T}_+\Sigma^{-1}\g^*\hat{\otimes}\h$ where $\hat{T}_+\Sigma^{-1}\g$ is considered as a differential graded associative algebra with differential induced by $m_\g$.

Similarly, if $(\g, m_\g)$ is an $L_\infty$-algebra and $\h$ is a differential graded Lie algebra, then an $L_\infty$-morphism $\g\to \h$ is nothing but a Maurer-Cartan element in the differential graded Lie algebra 	$\hat{S}_+\Sigma^{-1}\g^*\hat{\otimes}\h$ where $\hat{S}_+\Sigma^{-1}\g^*$ is considered as a differential graded commutative algebra with differential induced by $m_\g$.
 	
This leads to the following result.
\begin{thm}\label{thm:definfinity}Let $\g$ be a graded vector space: \begin{enumerate}
		\item
		Let $m\in \Der(\hat{T}\Sigma^{-1}\g^*)$ be an $A_\infty$-structure on  $\g$, $\h$ be a differential graded associative algebra and $f$ be a Maurer-Cartan element in the differential graded associative algebra $\hat{T}_+\Sigma^{-1}\g^*\hat{\otimes}\h$ determining an $A_\infty$-morphism $\g\to \h$. Then
		for an element $f^\prime\in \hat{T}_+\Sigma^{-1}\g^*\hat{\otimes}\h$ of degree $1$, the element $f+f^\prime\in \hat{T}_+\Sigma^{-1}\g^*\hat{\otimes}\h$ gives another $A_\infty$-morphism $\g\to\h$ if and only if $f^\prime$ is a Maurer-Cartan element in the
		differential graded associative  algebra $(\hat{T}_+\Sigma^{-1}\g^*\hat{\otimes}\h, d_f)$.	
		\item Let $m\in \Der(\hat{S}\Sigma^{-1}\g^*)$ be an $L_\infty$-structure on  $\g$, $\h$ be a differential graded Lie algebra and $f$ be a Maurer-Cartan element in the differential graded Lie algebra $\hat{S}_+\Sigma^{-1}\g^*\hat{\otimes}\h$ determining an $L_\infty$-morphism $\g\to \h$. Then
		for an element $f^\prime\in \hat{S}_+\Sigma^{-1}\g^*\hat{\otimes}\h$ of degree $1$, the element $f+f^\prime\in \hat{S}_+\Sigma^{-1}\g^*\hat{\otimes}\h$ gives another $L_\infty$-morphism $\g\to\h$ if and only if $f^\prime$ is a Maurer-Cartan element in the
		differential graded  Lie algebra $(\hat{L}_+\Sigma^{-1}\g^*\hat{\otimes}\h, d_f)$.
	\end{enumerate}
\end{thm}
\begin{rem}
	The reader may wonder whether Theorem \ref{thm:definfinity} can be extended to the $A_\infty$ or $L_\infty$-morphisms for which the target is a genuine $A_\infty$ or $L_\infty$-algebra. The subtlety here is that an $A_\infty$-morphism between two $A_\infty$-algebras is not a Maurer-Cartan element in a differential graded associative algebra but in an $A_\infty$-algebra. Similarly, an $L_\infty$-morphism is, in general, a Maurer-Cartan element in an $L_\infty$-algebra. Thus, one needs to set up a formalism of Maurer-Cartan elements in $A_\infty$ and $L_\infty$-algebras. This can be done, cf. for example \cite{Lazarev}. Once this formalism is in place, an appropriate analogue of Theorem \ref{thm:definfinity} will hold.
\end{rem}

\emptycomment{\begin{rem}
 Since every $L_\infty$-algebra can be obtained via higher derived brackets \cite{Vo}. So, when we want to study deformation theory of mathematical structures we should use higher derived brackets to construct $L_\infty$-algebras which control the deformation theory.
\end{rem}}

An $A_\infty$ or $L_\infty$-morphism whose target is, respectively, a differential graded associative algebra or a differential graded Lie algebra specializes to the notion of a left $A_\infty$-module or an $L_\infty$-representation. Note that if $V$ is a differential graded vector
space then its algebra of endomorphisms $\End(V)$ is naturally a differential graded associative algebra and a differential graded Lie algebra.
\begin{defi}
	Let $\g$ be a graded vector space:\begin{enumerate}
		\item Let $(\g, m_\g)$ be an $A_\infty$-structure on $\g$. The structure of a left $A_\infty$-module over $(\g, m_\g)$ on a differential graded vector space $V$ is an $A_\infty$-morphism $\g\to \End(V)$.
		\item
		Let $(\g, m_\g)$ be an $L_\infty$-structure on $\g$. The structure of an $L_\infty$-representation of $(\g, m_\g)$  on a differential graded vector space $V$ is an $L_\infty$-morphism $\g\to \End(V)$.
	\end{enumerate}
\end{defi}
The following is an immediate consequence of Theorem \ref{thm:definfinity}.
\begin{cor}
	Let $\g$ be a graded vector space and $V$ be a differential graded vector space: \begin{enumerate}
		\item
		Let $m\in \Der(\hat{T}\Sigma^{-1}\g^*)$ be an $A_\infty$-structure on  $\g$ and $f$ is a Maurer-Cartan element in the differential graded algebra $\hat{T}_+\Sigma^{-1}\g^*\hat{\otimes}\End(V)$ determining a left $A_\infty$-module structure on $V$ over $\g$. Then
		for an element $f^\prime\in \hat{T}_+\Sigma^{-1}\g^*\hat{\otimes}\End(V)$ of degree $1$, the element $f+f^\prime\in \hat{T}_+\Sigma^{-1}\g^*\hat{\otimes}\End(V)$ gives another left $A_\infty$-module structure on $V$ over $\g$  if and only if $f^\prime$ is a Maurer-Cartan element in the
		differential graded associative algebra $(\hat{T}_+\Sigma^{-1}\g^*\hat{\otimes}\End(V), d_f)$.	
		\item Let $m\in \Der(\hat{S}\Sigma^{-1}\g^*)$ be an $L_\infty$-structure on  $\g$ and $f$ is a Maurer-Cartan element in the differential graded Lie algebra $\hat{S}_+\Sigma^{-1}\g^*\hat{\otimes}\End(V)$ determining an $L_\infty$-representation structure on $V$ of $\g$. Then
		for an element $f^\prime\in \hat{T}_+\Sigma^{-1}\g^*\hat{\otimes}\End(V)$ of degree $1$, the element $f+f^\prime\in \hat{S}_+\Sigma^{-1}\g^*\hat{\otimes}\End(V)$ gives another $L_\infty$-representation structure on $V$ of $\g$ if and only if $f^\prime$ is a Maurer-Cartan element in the
		differential graded Lie algebra $(\hat{S}_+\Sigma^{-1}\g^*\hat{\otimes}\End(V), d_f)$.
	\end{enumerate}
\end{cor}
\bigskip
\noindent{\bf Acknowledgements.} This research was partially supported by NSFC (11922110).


\setlength{\parindent}{0pt}
\end{document}